\documentclass[a4paper, 11pt, reqno]{amsart}
\usepackage{etex}
\usepackage{amssymb}
\usepackage{amsthm}
\usepackage{bm}
\usepackage{latexsym}
\usepackage{amsmath}
\usepackage{eufrak}
\usepackage{mathrsfs}
\usepackage{amscd}
\usepackage[all]{xy}
\usepackage[usenames]{color}
\usepackage{graphicx}
\usepackage{color}
\usepackage{caption}
\usepackage{amscd}
\theoremstyle{plain}
\usepackage[usenames,dvipsnames]{pstricks}
\usepackage{epsfig}
\usepackage{pst-grad} 
\usepackage{pst-plot} 
\usepackage[space]{grffile} 
\usepackage{etoolbox} 
\makeatletter 
\patchcmd\Gread@eps{\@inputcheck#1 }{\@inputcheck"#1"\relax}{}{}
\makeatother
\usepackage{longtable}
\usepackage{tikz}
\usetikzlibrary{arrows}
\usetikzlibrary{decorations.pathreplacing}

\usepackage{hyperref}

\usepackage[top=1.1in, left=1.1in, right=1.1in, bottom=1.1in]{geometry}

\usepackage{longtable}
\usepackage{overpic}

\numberwithin{equation}{section}

\newtheorem{theorem}{Theorem}[section]
\newtheorem{proposition}[theorem]{Proposition}
\newtheorem{lemma}[theorem]{Lemma}
\newtheorem{corollary}[theorem]{Corollary}
\newtheorem{conjecture}[theorem]{Conjecture}

\theoremstyle{definition}

\newcommand{\appsection}[1]{\let\oldthesection\thesection
\renewcommand{\thesection}{Appendix \oldthesection}
\section{#1}\let\thesection\oldthesection}

\newtheorem{definition}[theorem]{Definition}

\theoremstyle{remark}

\newtheorem{remark}[theorem]{Remark}
\newtheorem{example}[theorem]{Example}

\DeclareMathOperator{\Amp}{Amp}
\DeclareMathOperator{\Nef}{Nef}
\DeclareMathOperator{\BBig}{Big}
\DeclareMathOperator{\Eff}{Eff}
\DeclareMathOperator{\Bir}{Bir}
\DeclareMathOperator{\Aut}{Aut}
\DeclareMathOperator{\Mov}{\overline{Mov}}
\DeclareMathOperator{\Pic}{Pic}

\DeclareMathOperator{\codim}{codim}

\def\R{{\mathbb{R}}}
\def\Z{{\mathbb{Z}}}

\def\C{{\mathbb{C}}}
\def\P{{\mathbb{P}}}

\renewcommand{\O}{{\mathcal{O}}}

\renewcommand{\to}[1][]{\xrightarrow{\ #1\ }}

\begin{document}
\title[Movable cone theorem for Calabi-Yau complete intersections]{Birational automorphism groups and the movable cone theorem for Calabi-Yau complete intersections of products of projective spaces}
\author[Jos\'e Y\'a\~nez]{Jos\'e Ignacio Y\'a\~nez}

\address{Department of Mathematics, University of Utah, 155 South 1400 East, Salt Lake City, UT 84112, USA.}

\email{yanez@math.utah.edu}

%
%
%

\begin{abstract}

For a Calabi-Yau manifold $X$, the Kawamata -- Morrison movable cone conjecture connects the convex geometry of the movable cone $\Mov(X)$ to the birational automorphism group. Using the theory of Coxeter groups, Cantat and Oguiso proved that the conjecture is true for general varieties of Wehler type, and they described explicitly $\Bir(X)$. We generalize their argument to prove the conjecture and describe $\Bir(X)$ for general complete intersections of ample divisors in arbitrary products of projective spaces. Then, under a certain condition, we give a description of the boundary of $\Mov(X)$ and an application connected to the numerical dimension of divisors.

\end{abstract}
\maketitle
\tableofcontents

\section{Introduction} \label{section:Introduction}

A Calabi-Yau manifold is a smooth projective manifold $X$ such that $\O_X(K_X) \simeq \O_X$ and $H^1(X,\O_X) = 0$. 

When studying the geometry of a variety $X$, the N\'eron-Severi group $N^1(X)$ of divisors modulo numerical equivalence plays a major role. In particular, certain cones inside the real vector space $N^1(X)_\R$ can give precise information about geometric properties of $X$. Given, for example, that the automorphism group of $X$ acts naturally on the nef cone, one can ask if the structure of $\Nef(X)$ is related to the complexity of $\Aut(X)$. These kinds of questions are addressed by the Kawamata-Morrison conjecture.

\begin{conjecture}[\cite{Morrison93},\cite{Kawamata97}]
Let $X$ be a Calabi-Yau manifold.
\begin{enumerate}
\item[(a)] There exists a rational polyhedral cone $C$ which is a fundamental domain for the action of $\Aut(X)$ on $\Nef^e(X):= \Nef(X) \cap \Eff(X)$. More explicitely, \[ \Nef^e(X) = \bigcup_{g\in \Aut(X)} g^* C ,\] and $\mathrm{int}\, C \cap \mathrm{int}\, g^*C=\emptyset$, unless $g^* = \mathrm{id}$.
\item[(b)] There exists a rational polyhedral cone $C$ which is a fundamental domain for the action of $\Bir(X)$ on ${\Mov\,}^e(X):= \Mov(X) \cap \Eff(X)$. More explicitely, \[ {\Mov\,}^e(X) = \bigcup_{g\in \Bir(X)} g^* C ,\] and $\mathrm{int}\, C \cap \mathrm{int}\, g^*C=\emptyset$, unless $g^* = \mathrm{id}$.
\end{enumerate}
\label{ConjectureKM}
\end{conjecture}

Roughly speaking, the conjecture says that if the nef cone (movable cone) of $X$ is not rational polyhedral, then the (birational) automorphism group is infinite. On the other hand, if the (birational) automorphism group is finite, then the nef cone (movable cone) has to be rational polyhedral.

The conjecture has been proven in several cases (see \cite{LazicOguisoPeternell18} for an in-depth discussion of the conjecture). In \cite{CantatOguiso15}, Cantat and Oguiso prove Conjecture \ref{ConjectureKM} for a series of examples. Let $X$ be a hypersurface of multidegree $(2,\ldots,2)$ in $(\P^1)^{n+1}$, with $n\geq 3$. Then $X$ is a Calabi - Yau manifold of dimension $n$, and the following theorem holds. 

\begin{theorem}[\cite{CantatOguiso15}]
Assume that $X$ is given by a (very) general divisor of multidegree $(2,\ldots,2)$. Then:
\begin{enumerate}
\item[(a)] $\Aut(X)$ corresponds to a finite subgroup of $\Aut((\P^1)^{n+1})$, and it acts trivially on $N^1(X)_\R$. For a very general choice of the divisor, $\Aut(X)$ is trivial.
\item[(b)] $\Bir(X)$ is isomorphic to $\Aut(X)\cdot (\underbrace{\Z/2\Z \ast \cdots \ast \Z/2\Z}_{n+1})$.
\item[(c)] Conjecture \ref{ConjectureKM} holds for $X$. More precisely, $\Nef(X)$ is rational polyhedral and it is a fundamental domain for the action of $\Bir(X)$.
\end{enumerate}
\label{thmCO15}
\end{theorem}

Let $\P:= \P^{n_1}\times \ldots \times \P^{n_l}$ such that $\sum n_i \geq 4$, and if $l=2$ then $(n_1,n_2)\neq (2,2)$. Let $X$ be a complete intersection Calabi-Yau subvariety in $\P$ given by the intersection of $m$ ample divisors, with $m\leq \min\{n_i\}$. In this article we generalize Theorem \ref{thmCO15} to this context.

\begin{theorem}
Assume that $X$ as above is given by a general choice of ample divisors. Then:
\begin{enumerate}
\item[(a)] $\Aut(X)$ corresponds to a finite subgroup of $\Aut(\P)$, and it acts trivially on $N^1(X)_\R$.
\item[(b)] $\Bir(X)$ is isomorphic to $\Aut(X)\cdot (\underbrace{\Z/2\Z \ast \cdots \ast \Z/2\Z}_{|J|})$, where \[ J = \{j\mid n_j = \min\{n_i\} = \codim X\}. \]
\item[(c)] Conjecture \ref{ConjectureKM} holds for $X$. More precisely, $\Nef(X)$ is rational polyhedral and it is a fundamental domain for the action of $\Bir(X)$.
\end{enumerate}
\label{mainThm}
\end{theorem}

This result also includes the results from \cite[Chapter 8]{Skauli17}, \cite[Proposition 1.4]{Oguiso14}, \cite[Theorem 1.4]{Oguiso18}, \cite[Example 3.8 (4)]{Kawamata97} and \cite[Theorem 1.1 (ii) when $d=2$]{Ottem15}. The case when $l=2$ and $(n_1,n_2) = (2,2)$ was studied by Silverman in \cite{Silverman91}.

As in \cite{CantatOguiso15}, the main tool we use is Coxeter groups. These groups, along with some results about Calabi-Yau manifolds, are introduced in Section \ref{section:Preliminaries}. In Section \ref{section:ComputingBir} we prove parts (a) and (b) by realizing $\Bir(X)$ as a subgroup of a Coxeter group. In Section \ref{section:MovableCone} we prove part (c) by establishing an isomorphism between the nef cone of $X$ and the fundamental domain of the action of the Coxeter group. 

Cantat and Oguiso give a description of the boundary of $\Mov(X)$ in terms of what they call ``chinese hat shells'' (See \cite[Sections 2.2.6, 2.2.7 and 2.2.8]{CantatOguiso15}). This view still works in our case when $n=1$. We give a description of the boundary of $\Mov(X)$ in terms of cones involving the eigenvectors of elements of $\Bir(X)$. See Figure \ref{figureMovable} for an example. 

\begin{theorem}
Let $\{h_1,\ldots,h_l\}$ be the set of extremal rays of the simplicial cone $\Nef(X)$ and suppose that $\Bir(X)$ is Lorentzian, seen as a subgroup of a Coxeter group. Then the boundary of the movable cone $\Mov(X)$ is the closure of the union of the following sets:
\begin{enumerate}
\item The $\Bir(X)$-orbit of the codimension one faces $\{ \sum_{k\neq i} a_kh_k \mid a_k\geq 0\}$, for $i\not \in J$; and
\item The $\Bir(X)$-orbit of the cones $\{ a_\lambda v_\lambda + \sum_{k\neq i,j} a_kh_k \mid a_k\geq 0, a_\lambda\geq 0\}$, with $i,j\in J$, where $v_\lambda$ is:
\begin{itemize}
\item If $n \geq 2$, an eigenvector associated to the unique eigenvalue $\lambda > 1$ of $(\iota_i\iota_j)^*$; or
\item If $n = 1$, $v_\lambda = 0$.
\end{itemize}
\end{enumerate}
\label{mainThm2}
\end{theorem}

\begin{figure}[ht]
  \begin{center}
  \hfill
  \begin{minipage}[c]{.45\textwidth}
  \includegraphics[scale=.7]{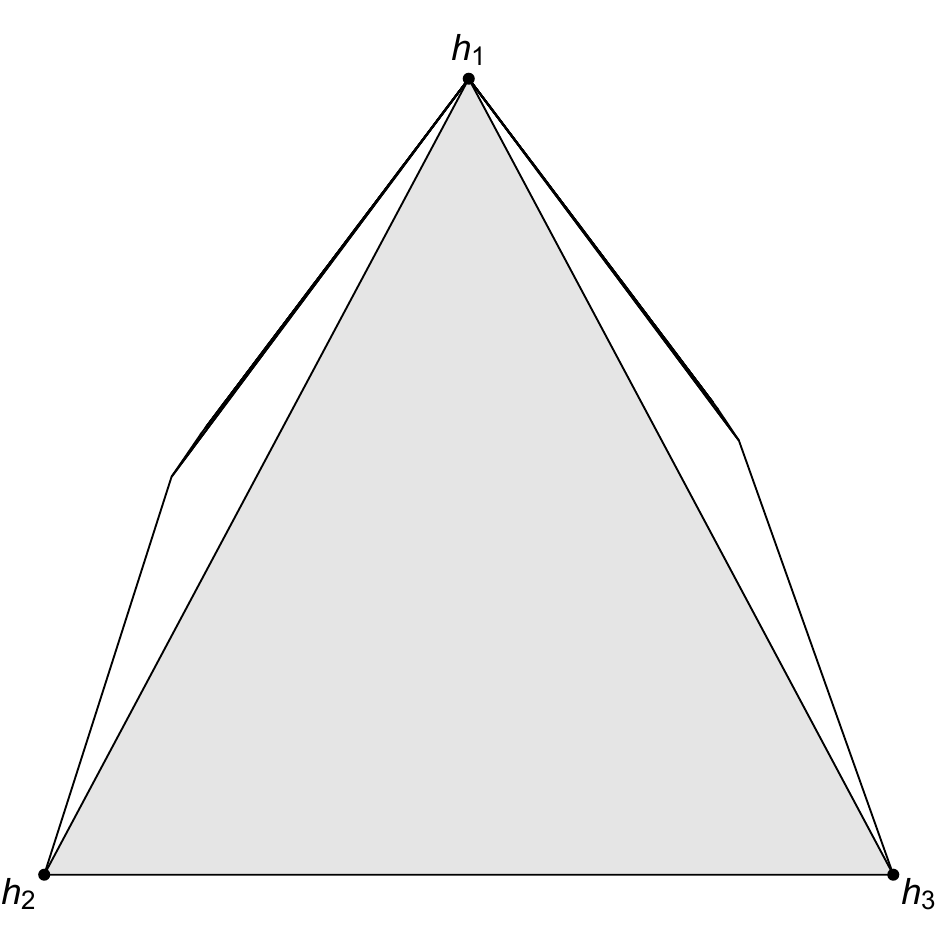}
  \end{minipage} \hfill \hfill \begin{minipage}[c]{.45\textwidth}
  \includegraphics[scale=.7]{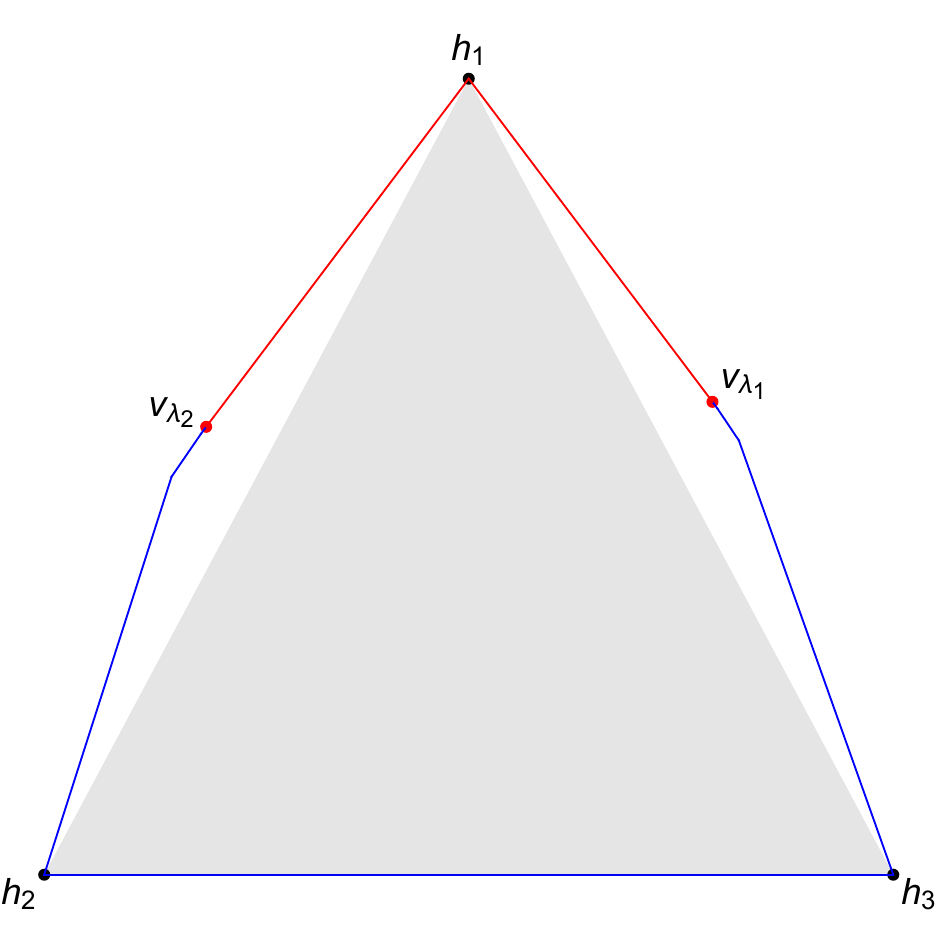}
  \end{minipage}
  \hfill
  \end{center}
  \caption{Example of the movable cone for $X$ given by the intersection of divisors of multidegree $(2,2,1)$ and $(2,1,2)$ in $\P^3\times \P^2 \times \P^2$. In this case $J = \{2,3\}$. On the left, the Movable cone obtained as the $\Bir(X)$-orbit of $\Nef(X)$ (in gray). Notice that the top right and top left thicker lines correspond to accumulations of chambers. On the right, the boundary of the Movable cone computed using Theorem \ref{mainThm2}. In blue, the cones corresponding to Theorem \ref{mainThm2} (1), and in red, the cones corresponding to Theorem \ref{mainThm2} (2).}
  \label{figureMovable}
\end{figure}

Finally, in Section \ref{section:NumericalDimension} we give a small application for these examples, expanding the possible values of the numerical dimension $\nu_{\text{vol}}^\R$ introduced in \cite{Lesieutre19} and \cite{Mccleerey19}. Lesieutre used Oguiso's example \cite[Proposition 1.4]{Oguiso14} to produce a divisor class with a numerical dimension $\kappa_\sigma^\R$ and $\nu_{\text{vol}}^\R$ equal to 3/2, giving a counterexample to the notion that all numerical dimension of divisors coincide. We show that if $|J| \geq 3$, then $X$ has infinitely many divisor classes with numerical dimension $\nu_{\text{vol}}^\R$ equal to $(\dim X)/2$, and if Conjecture \ref{conjectureLorentzian} is true, then this is the only possible value that can be obtained from eigenvectors.

\section*{Acknowledgements}
The author would like to thank Christopher Hacon for useful comments, discussions and constant support. The author is also grateful to John Lesieutre for valuable comments and suggestions.
This research was partially supported by NSF research grants no:  DMS-1952522, DMS-1801851 and by a grant from the Simons Foundation; Award Number: 256202.

\section{Preliminaries} \label{section:Preliminaries}

In this section we give some definitions and results that will be needed in this paper.

\subsection{Coxeter groups} \label{subsection:CG}

For a complete introduction to Coxeter groups we refer to \cite{Humphreys90}. Let $W$ be a finitely generated group and let $S = \{s_i\}_{i=1}^n$ be a finite set of generators. We say that the pair $(W,S)$ is a \emph{Coxeter system} if there are integers $m_{ij} \in \Z_{\geq 0} \cup \{\infty\}$ such that \begin{itemize}
\item $W = \langle s_i \mid (s_is_j)^{m_{ij}} = 1\rangle$,
\item $m_{ii} = 1$ for all $i$,
\item $m_{ij} = m_{ji} \geq 2$ or $= \infty$.
\end{itemize}

A group is called a \emph{Coxeter group} if there is a finite subset $S\subset W$ such that $(W,S)$ is a Coxeter system.

Let $J\subseteq \{1,\ldots,n\}$. Define $W_J$ as the subgroup of $W$ of all elements that can be generated by the set $\{s_j\mid j\in J\}$.

We associate a matrix $B$ to $(W,S)$ as follows:
\[ (B)_{ij} = \begin{cases} -\cos \frac{\pi}{m_{ij}} & \text{if } m_{ij} < \infty \\ -c_{ij} & \text{if } m_{ij} = \infty \end{cases} \] where the $c_{ij}$ are arbitrarily chosen real numbers such that $c_{ij} = c_{ji} \geq 1$. Classically, $c_{ij} = 1$, but after the work of Vinberg \cite{Vinberg71}, we can have $c_{ij} \geq 1$.

Let $V$ be a real vector space of dimension $n$, with basis $\Delta = \{\alpha_s\}_{s\in S}$ and denote $\alpha_{s_i}=\alpha_i$. We can define a bilinear form $\mathcal{B}$ on $V$ by $\mathcal{B}(\alpha_i,\alpha_j) = \alpha_i^T B \alpha_j = (B)_{ij}$, for $s_i,s_j\in S$. For a vector $v\in V$ define \[ \sigma_v(w) = w - 2\dfrac{\mathcal{B}(w,v)}{\mathcal{B}(v,v)}v. \] Then the homomorphism $\rho\colon W \to GL(V)$ defined by $s\mapsto \sigma_{\alpha_s}$ is a faithful geometric representation of the Coxeter group $W$. (For the classic case, see \cite[Section 5.4 Corollary]{Humphreys90}. For the general case, this is \cite[Theorem 2 (6)]{Vinberg71}.)

Define in the dual space $V^*$ the covex cone $D \subset V^*$ as the intersection of the half-spaces \[ D_s^+ := \{f\in V^* \mid f(\alpha_s) \geq 0 \} \] for all $s\in S$.

Denote the dual representation of $W$ on $V^*$ as $\rho^*$. Then we define the \emph{Tits cone $T$} as the union \[ T := \bigcup_{w\in W} \rho^*(w)(D). \] From \cite[Theorem 2 and Proposition 8]{Vinberg71} we have that $T$ is a convex cone and $D$ is a fundamental domain for the action of $W$.

Throughout this article we assume that $m_{ij} = \infty$ if $i\neq j$, and that $c_{ij}\in \R_{\geq 1}$. In this situation, the group $W$ is isomorphic to $\underbrace{\Z/2\Z \ast \ldots \ast \Z/2\Z}_{n}$, and the $n\times n$ matrix $M(n)_j$ associated to $\sigma_{s_j}$ with respect to the basis $\{\alpha_i\}$ is \[ M(n)_j = \begin{pmatrix}
1 & 0 & \cdots & 0 & 0 & 0 & \cdots & 0 \\ 
0 & 1 & \cdots & 0 & 0 & 0 & \cdots & 0 \\ 
\vdots & \vdots & \ddots & \vdots & \vdots & \vdots & \cdots & \vdots \\ 
0 & 0 & \cdots & 1 & 0 & 0 & \cdots & 0 \\ 
2c_{j\,1} & 2c_{j\,2} & \cdots & 2c_{j\,{j-1}} & -1 & 2c_{j\,{j+1}} & \cdots & 2c_{j\, n} \\ 
0 & 0 & \cdots & 0 & 0 & 1 & \cdots & 0 \\ 
\vdots & \vdots & \ddots & \vdots & \vdots & \vdots & \ddots & \vdots \\ 
0 & 0 & \cdots & 0 & 0 & 0 & \cdots & 1
\end{pmatrix}  \] Notice that the geometric representation depends on the matrix $B$, so we keep track of this by denoting our Coxeter system as $(W,S)_B$.
~\\

When studying the boundary of $T$ it is convenient to work on $\P V$ as follows. Choose $V_1$ a hyperplane such that each ray $\R_{>0} \alpha$, for $\alpha\in \Delta$, intersects at one point and denote it $\hat{\alpha}$. Let $\varphi$ be the linear form such that $V_1$ is $\varphi(v) = 1$. Define \[ \hat{v} = \frac{v}{\varphi(v)},\ \text{for }v\in V\setminus \{\varphi(v)=0\}. \] We will refer to points in $\P V$ as \emph{directions} of $V$.

For Theorem \ref{mainThm2} we need to study some properties of $W_J$. We say that the subgroup $W_J$ is \emph{Lorentzian} if the signature of the matrix $B_J$, obtained by removing the rows and columns with indices not in $J$, has signature $(|J|-1,1)$.

The main consequence of this property is that given an infinite reduced word $\mathbf{w} = s_{k_1}s_{k_2}\cdots$ in $W_J$, any injective sequence $\{w_i\cdot \widehat{x}\}$, where $w_i = s_{k_1}s_{k_2}\cdots s_{k_i}$, converges to the same direction, regardless of $\widehat{x} \in \P V$.

\begin{theorem}
Let $\mathbf{w}$ an infinite reduced word in a Lorentzian subgroup $W_J$. Then every injective convergent sequence $\{w_i\cdot \hat{x}\}$ converge to a same unique direction $\widehat{\gamma}(\mathbf{w})$.
\label{thmInfiniteWord}
\end{theorem}

\begin{proof}
This result follow from \cite[Theorem 2.5]{ChenLabbe17}, \cite[Corollary 2.6]{ChenLabbe17}, \cite[Theorem 2.8]{ChenLabbe17} and \cite[Corollary 2.9]{ChenLabbe17}.
\end{proof}

\subsection{Calabi-Yau manifolds} \label{subsection:CY}

We define a \emph{Calabi-Yau manifold} as a smooth projective manifold $X$ such that $\O_X(K_X) \simeq \O_X$ and $H^1(X,\O_X) = 0$. In particular, a Calabi-Yau manifold in the strict sense, \textit{i.e.} a smooth projective manifold with $\O_X(K_X) \simeq \O_X$ and $H^0(X,\Omega_X^k) = 0$ for $0<k<\dim X$, is a Calabi-Yau manifold.

Denote the N\'eron-Severi group, generated by the numerical classes of line bundles, by $N^1(X)$. The rank of $N^1(X)$, $\rho(X)$, is the \emph{Picard number of $X$}. Let $N^1(X)_\R$ the real vector space generated by $N^1(X)$. Let $\Amp(X)\subseteq N^1(X)_\R$ the open cone of ample divisors, $\Nef(X)$ the closed cone of Nef divisors, $\BBig(X)$ the open cone of big divisors, and $\Eff(X)$ the cone of effective divisors. As usual, $\overline{\Amp(X)} = \Nef(X)$ and $\overline{\BBig(X)} = \overline{\Eff(X)}$. This last cone is called the \emph{pseudoeffective cone of $X$}.

We say that a divisor is \emph{movable} if the base locus of $|D|$ has a codimension 2. The \emph{movable cone $\Mov(X)$} is the closure of the convex cone generated by the classes of movable divisors. Define the \emph{movable effective cone ${\Mov\,}^e(X)$} as $\Mov(X) \cap \Eff(X)$. Similarly, define $\Nef^e(X) = \Nef(X) \cap \Eff(X)$.

Let $\Aut(X)$ be the group of automorphisms of $X$, and $\Bir(X)$ the group of birational automorphisms. By Theorem \ref{thmKFlops} we have that every birational automorphism is an isomorphism in codimension 1. So if $g\in \Bir(X)$ and $D$ is movable, then $g^* D$ is again movable. Hence $\Bir(X)$ acts naturally on $\Mov(X)$, $\Eff(X)$ and ${\Mov\, }^e(X)$. This is the setup for the Kawamata-Morrison conjecture (Conjecture \ref{ConjectureKM}).

We need a few more results related to Calabi-Yau manifolds:

\begin{theorem}[\cite{Kawamata88}, Theorem 5.7]
Any flopping contraction of a Calabi-Yau manifold is given by a codimension one face of $\Nef(X)$ up to automorphisms of $X$.
\label{thmKContraction}
\end{theorem}

\begin{theorem}[\cite{Kawamata08}, Theorem 1]
Every element of $\Bir(X)$ can be decomposed as a sequence of flops, and an automorphism of $X$. That is, for every $g\in \Bir(X)$ there is a sequence of flops $\gamma_i \colon X_i \dashrightarrow X_{i-1}$ such that \[ g = f\circ \gamma_n\circ \ldots \circ \gamma_1, \] with $X_0 = X_n = X$, and $f\in \Aut(X)$.
\label{thmKFlops}
\end{theorem}

\begin{theorem}[\cite{CantatOguiso15}, Theorem 3.1]
Let $n\geq 3$ be an integer and $V$ be a Fano manifold of dimension $(n+1)$. Let $M$ be a smooth member of the linear system $|-K_V|$. Let $\eta\colon M\to V$ be the natural inclusion. Then:
\begin{enumerate}
\item $M$ is a Calabi-Yau manifold of dimension $n\geq 3$.
\item The pull-back morphism $\eta^*\colon \Pic(V)\to \Pic(M)$ is an isomorphism, and it induces an isomorphism of Nef cones:
\[ \eta^*(\Nef(V)) \simeq \Nef(M). \]
\item $\Aut(M)$ is a finite group.
\end{enumerate}
\label{thmCYFano}
\end{theorem}

\begin{proof}
By the adjunction formula we have that $\O_M(K_M) \simeq \O_M$, and by the Lefschetz hyperplane theorem, $\pi_1(M) \simeq \pi_1(V) = \{1\}$, because $V$ is Fano, and so simply connected. Now, taking the long exact sequence in cohomology induced by the exact sequence \[ 0 \to \O_V(K_V) \to \O_V \to \O_M \to 0\] we get that $h^{k}(\O_M) = 0$ for $0<k <n$, and by Hodge symmetry we obtain that $h^0(\Omega^k_M) = 0$ for $0<k<n$. This proves (1).

The first part of (2) follows from Lefschetz hyperplane theorem, and the second follows from a result of Koll\'ar \cite[Appendix]{Borcea89} which says that the natural map $\eta_*\colon \overline{NE}(M) \to \overline{NE}(V)$ is an isomorphism. The result follows by taking dual cones.

For (3), notice that part (1) implies that $h^0(T_M)= 0$ and so $\dim \Aut(M) = 0$. Given that $\Nef(M)$ is rational polyhedral, there exists a divisor class $h$ that is fixed by $\Aut(M)$. Consider the embedding $M\to \P(H^0(M,L^{\otimes k})^\vee)$, with $L$ a line bundle such that $c_1(L) = h$. We then have that the image of $\Aut(M)$ is the algebraic subgroup of $\mathrm{PGL}(H^0(M,L^{\otimes k})^\vee)$ that preserves the image of $M$. Then $\Aut(M)$ is finite because $\dim \Aut(M) = 0$.
\end{proof}

\subsection{Calabi-Yau complete intersection of products of projective spaces} \label{subsection:CYCI} \mbox{}

Let $l\geq 2$, and $\{n_1, n_2 , \ldots , n_l\}$ be integers greater or equal to 1. In order to apply Theorem \ref{thmCYFano} we require that $\sum n_i \geq 4$, and that if $l=2$, then $(n_1,n_2)\neq (2,2)$. Define

\begin{align*}
P(n_1,\ldots, n_l) &:=\  \P^{n_1} \times\cdots \times \P^{n_l}\\
P(n_1,\ldots, n_l)_j &:=\  \P^{n_1} \times\cdots \times \P^{n_{j-1}} \times \P^{n_{j+1}} \times \cdots \times \P^{n_l}
\end{align*}
and denote the natural projections as 
\begin{align*}
p^j: &\ P(n_1,\ldots,n_l) \to \P^{n_j}\\
p_j: &\ P(n_1,\ldots, n_l) \to P(n_1,n_2,\ldots, n_l)_j
\end{align*}
Let $H_j$ be the divisor class of $(p^j)^*(\O_{\P^{n_j}}(1))$. Then $P(n_1,\ldots,n_l)$ is a Fano manifold of dimension $\sum n_i$ that satisfies
\begin{align*}
N^1(P(n_1,\ldots,n_l)) &= \bigoplus_{j=1}^{l} \R H_j,\\
-K_{P(n_1,\ldots, n_l)} &= \sum_{j=1}^l (n_j + 1)H_j,\\
\Nef(P(n_1,\ldots,n_l)) &= \left\{ \sum a_iH_i\mid a_i \in \R_{\geq 0}\text{ for all }i\right\}
\end{align*}

Set $n := \min\{n_i\}$ and $J = \{j \mid n_j = n\}$.

Let $X$ be the complete intersection of general ample divisors $D_1, \ldots, D_m$, with $m\leq n$, such that $\sum D_i = -K_{P(n_1,\ldots, n_l)}$. 

\begin{proposition}
$X$ is a Calabi-Yau manifold with $\rho(X) = l$ and $\Nef(X)$ is a simplicial cone generated by $h_j := \eta^*(H_j)$, where $\eta\colon X \to P(n_1,\ldots, n_l)$ is the natural inclusion.
\end{proposition}

\begin{proof}
The divisors $D_1,\ldots, D_m$ are chosen generally, so by Bertini's theorem we get that $X$, and all partial intersections $D_{i_1}\cap D_{i_2}\cap \cdots \cap D_{i_r}$ are smooth. By the adjunction formula we get that each partial intersection is a Fano manifold. The conditions that $\sum n_i \geq 4$ and that if $l=2$, then $(n_1,n_2) \neq (2,2)$, along with the restriction on the codimension of $X$, implies that we can use Theorem \ref{thmCYFano} to obtain that $\eta^*(\Nef(\P)) \simeq \Nef(X)$.
\end{proof}

Define the projections $\pi_j:= p_j\circ \eta\colon X \to P(n_1,\ldots,n_l)_j$. By taking general divisors, these projections are surjective, and if $n_j = n$, then $\pi_j$ defines a birational involution $\iota_j$. This follows from the following result.

\begin{proposition}
If $n_j=n$ and $\codim X = n$, then $\pi_j$ is generically 2-to-1.
\end{proposition}

\begin{proof}
Write $D_i = \sum_{k=1}^l a_{ik}H_k$, where $a_{ik}\geq 1$. Since $\sum D_i = -K_{P(n_1,\ldots,n_l)}$, we have that $\sum_{i=1}^n a_{ij} = n_j  + 1$. When $n_j = n$ this means that one, and only one, $a_{ij}$ has to be equal to 2 (and all the other $a_{ij}$ are equal to 1). Because $\dim X = \dim P(n_1,\ldots, n_l)_j$ we get that the projection $\pi_j$ is generically 2-to-1.
\end{proof}

We write $\P(\mathbf{n})$ instead of $P(n_1,\ldots,n_l)$ when there is no possible confusion.

\section{Computing \texorpdfstring{$\Bir(X)$}{Bir(X)}} \label{section:ComputingBir}

The goal of this section is to prove Theorem \ref{mainThm} (a) and (b). More precisely, we will show that the group generated by the involutions $\iota_j$, with $j\in J$, is isomorphic to $W_J$, where $(W,S)$ is the Coxeter system given by $S = \{s_i\}_{i=1}^l$ and $m_{ij} = \infty$.

Throughout this section we assume that $\codim X = n$. See Remark \ref{remarkCodimensionBir} and Remark \ref{remarkCodimensionMov} for the case $\codim X < n$.

We begin by showing that the projections $\pi_j = p_j\circ \eta \colon X\to \P_j$, where $n_j=n$, are small, so that the non-finite locus of $\pi_j$ has codimension at least 2 in $X$.

\begin{proposition}[\cite{Skauli17}, Proposition 8.3.3] 
The maps \[ \pi_j \colon X\to\P_j,\] where $n_j = n$, are small.
\end{proposition}

\begin{proof}
We recall the proof for the reader's convenience. Without loss of generality, assume $j=l$.

Let $X = D_1 \cap \cdots \cap D_n$, and assume for simplicity, that $a_{nl} = 2$, \emph{i.e.} that $D_n$ is the divisor that has degree 2 on $\P^{n_l}$. Let $x_0,\ldots, x_n$ the coordinates of $\P^{n_l}$.

If $n=1$ we get that $X = f_1(\mathbf{y})x_0^2 + f_2(\mathbf{y})x_0x_1 + f_3(\mathbf{y})x_1^2 =0$, where $\mathbf{y}$ are coordinates of $\P(\mathbf{n})_l$. For a fixed point $\mathbf{y}_0\in \P(\mathbf{n})_l$ we get that $\pi_l^{-1}(\mathbf{y}_0) \simeq \P^1$  if and only if $x\in \{f_1 = f_2=f_3 =0\}$.We can assume the $f_i$'s are general so the set $\{f_1 = f_2=f_3 =0\}$ has codimension $\geq 3$, implying that $\pi_l$ contracts no divisor.

We need then to prove the case $n \geq 2$. In this case, $X$ is defined by $n-1$ linear equations and one quadratic equation in the variables $x_0,\ldots,x_n$. More precisely, $X$ is determined by the equations
\begin{eqnarray*}
G_1(\mathbf{x},\mathbf{y}):= g_{1,0}(\mathbf{y})x_0 + g_{1,1}(\mathbf{y})x_1 + \cdots + g_{1,n}(\mathbf{y})x_n &=&0\\
G_2(\mathbf{x},\mathbf{y}):= g_{2,0}(\mathbf{y})x_0 + g_{2,1}(\mathbf{y})x_1 + \cdots + g_{2,n}(\mathbf{y})x_n &=&0\\
\vdots \\
G_{n-1}(\mathbf{x},\mathbf{y}):= g_{n-1,0}(\mathbf{y})x_0 + g_{n-1,1}(\mathbf{y})x_1 + \cdots + g_{n-1,n}(\mathbf{y})x_n &=&0\\
F(\mathbf{x},\mathbf{y}) := \sum_{0\leq i \leq j \leq n} f_{ij}(\mathbf{y})x_ix_j &=& 0
\end{eqnarray*} where $\mathbf{y}$ are coordinates of $\P(\mathbf{n})_l$.

Let $V \subset \P(\mathbf{n})$ be the subvariety defined by the $G_i(\mathbf{x},\mathbf{y})$'s, and consider \[ \tau\colon V \to \P(\mathbf{n})_l \] the corresponding projection. Let $\mathbf{y}_0$ be a point in $\P(\mathbf{n})_l$. The fiber $\tau^{-1}(\mathbf{y}_0) \subset \P^n$ is given by 
\begin{eqnarray*}
g_{1,0}(\mathbf{y}_0)x_0 + g_{1,1}(\mathbf{y}_0)x_1 + \cdots + g_{1,n}(\mathbf{y}_0)x_n &=&0\\
g_{2,0}(\mathbf{y}_0)x_0 + g_{2,1}(\mathbf{y}_0)x_1 + \cdots + g_{2,n}(\mathbf{y}_0)x_n &=&0\\
\vdots \\
g_{n-1,0}(\mathbf{y}_0)x_0 + g_{n-1,1}(\mathbf{y}_0)x_1 + \cdots + g_{n-1,n}(\mathbf{y}_0)x_n &=&0
\end{eqnarray*} so the dimension of the fiber $\tau^{-1}(\mathbf{y}_0)$ is determined by the rank of the $(n-1)\times (n+1)$ matrix $[g_{i,j}(\mathbf{y}_0)]$: the fiber $\tau^{-1}(\mathbf{y}_0)$ has dimension $k$ if and only if $\mathrm{rk}\, [g_{i,j}(\mathbf{y}_0)] = n-k$.

Define $\mathrm{Fib}_k = \{\mathbf{y}_0 \in \P(\mathbf{n})_l \mid \dim \tau^{-1}(\mathbf{y}_0) \geq k\}$. Since the $g_{i,j}$ are chosen generally, it follows that $\codim \mathrm{Fib}_k = k^2 - 1$ (see \cite[Prop. 12.2]{Harris92} and \cite[Lemma 8.3.4]{Skauli17})

Then the locus of points where $\pi_l^{-1}(\mathbf{y})$ has dimension $k$ has codimesion greater or equal to $k^2 - 1$. So to check that the map is small, we need to check the cases when $k = 1$ and $k = 2$.

Notice that $\pi_l^{-1}(x)$ is given by $\tau^{-1}(\mathbf{y}) \cap (F(\mathbf{x},\mathbf{y}) = 0)$. A quadric containing a line is a codimension 3 condition, and a quadric containing a plane is a codimension 6 condition. Since the $D_i$'s are chosen generally, these codimensions are achieved \cite[Lemma 8.3.5]{Skauli17}.

Therefore the locus of points where $\pi_l^{-1}(\mathbf{y})$ has dimension 1 has codimension at least 3, and the locus of points where $\pi_l^{-1}(\mathbf{y})$ has dimension 2 has codimension at least 6, which finishes the proof.
\end{proof}

We now check that the action of $\Aut(X)$ is trivial on $N^1(X)$. For this, we use the following lemma.

\begin{lemma}[\cite{Skauli17}, Lemma 8.4.7] For all $j$, $H^0(X,\O_X(h_j)) \simeq H^0(\P(\mathbf{n}),\O_{\P(\mathbf{n})}(H_j))$.
\label{lemmaIsoSections}
\end{lemma}

\begin{proposition}[Theorem \ref{mainThm} (a)]
For a general choice of $D_1,\ldots, D_{n_l}$ the co-action of $\Aut(X)$ on $N^1(X)_\R$ is trivial.
\label{propAutAction}
\end{proposition}

\begin{proof}
By Theorem \ref{thmCYFano} (3), $\Aut(X)$ is finite. Let $\gamma \in \Aut(X)$. From Theorem \ref{thmCYFano} (2), $\eta^*(\Nef(\P(\mathbf{n})) \simeq \Nef(X) = \mathrm{cone}(h_1,\ldots,h_l)$. Because $\{h_1,\ldots,h_l\}$ are primitive generators of the extremal rays of $\Nef(X)$, and the elements of $\Aut(X)$ preserve the nef cone, we have that $\gamma$ preserves the set $\{h_1,h_2,\ldots,h_l\}$, and by Lemma \ref{lemmaIsoSections}, the action of $\gamma$ on the sections of $H^0(X,\O_X(h_j))$, for $1\leq j \leq l$, induces an automorphism $\gamma_{\P(\mathbf{n})}$ on $\P(\mathbf{n})$ such that $\gamma_{\P(\mathbf{n})}|_X = \gamma$.

Hence $\Aut(X)$ is a subgroup of $\Aut(\P(\mathbf{n})) = \left( \prod_{i=1}^l PGL_{n_i+1}(\C)\right) \rtimes H$, where $H$ is a subgroup of $\mathcal{S}_l$, the group of $l$ permutations, that depends on the $n_i$'s.

Let $G$ be the image of \[ \Aut(X) \to \Aut(\P(\mathbf{n})) \to H. \] If $g\in G$, with $g\neq \mathrm{id}$, there exists a lift to an automorphism of $\P(\mathbf{n})$ that restricts to an automorphism of $X$. If $g\neq \mathrm{id}$, this means that $g^*$ permutes the variables between the factors $\P^{n_i}$. Because the polynomials corresponding to the $D_j$ are chosen generally, it follows that this situation cannot happen, and so $G = \{\mathrm{id}\}$.

This implies that $\Aut(X)$ is a finite subgroup of $\prod_{i=1}^l PGL_{n_i+1}(\C)$. The co-action of $\prod_{i=1}^l PGL_{n_i+1}(\C)$ on $N^1(\P(\mathbf{n}))_\R$, and so on $N^1(X)_\R$, is trivial.

\end{proof}

\begin{remark}
It seems that one could replicate the proof of Theorem \ref{thmCO15} (1) to prove that $\Aut(X)$ is trivial for a very general choice of $X$, but it is not clear how to extend the computations done in \cite{CantatOguiso15} to our more general setup.
\end{remark}

To compare $\Bir(X)$ with $W_J$ we need to understand the action of the involutions $\iota_j$, with $j\in J$, on $N^1(X)_\R$.

Write the divisors $D_i$ with respect to the basis $\{H_j\}$. \[ D_i = \sum_{j=1}^{l} a_{ij} H_j. \] Consider $X$ as a codimension $n$ cycle in $\P(\mathbf{n})$ \[ X = (D_1\cdot \ldots \cdot D_n)_{\P(\mathbf{n})} \] and define the numbers $b_{ij}$ as the coefficient of $H_iH_j^{n-1}$ in the cycle $X$.

Define a matrix $B$ as follows 
\[(B)_{ij} = \begin{cases} 	1 & \text{if }i=j \\ 
							-b_{ij}/2 & \text{if } i\neq j \text{ and } j\in J \\
							-b_{ji}/2 & \text{if } i\neq j \text{ and } i\in J \\
							-n_l & \text{if } i\neq j \text{ and } i\not \in J \text{ and } j\not\in J \end{cases}\]

Notice that if $j \in J$ we use the coefficient $b_{ij}$, and when $i \in J$ we use the coefficient $b_{ji}$. Then we need to check that if both $i$ and $j$ are in $J$, the coefficients $b_{ij}$ and $b_{ji}$ are equal.

\begin{proposition}
The matrix $B$ is well-defined.
\label{propMatrixDefined}
\end{proposition}

\begin{proof}
We need to prove that $b_{ij} = b_{ji}$ when both $i$ and $j$ are in $J$. To do so, remember that the $a_{kj}$'s and the $a_{ki}$'s are all $1$, except for one case where it is equal to $2$. Then $b_{ij}$ corresponds to the coefficient of $xy^{n-1}$ in the expansion of $2(x+y)^{n}$ or $(2x^2+4xy+2y^2)(x+y)^{n-2}$. By symmetry, it is equal to the coefficient of $yx^{n-1}$, which corresponds to $b_{ji}$.
\end{proof}

\begin{proposition}
Let $j\in J$ and let $\iota_j$ be the corresponding birational involution. Then, with respect to the basis $\{h_i\}_{i=1}^l$ the matrix of $\iota_j^*$ is $M(l)_j^T$, with $M(l)_j$ be the matrix described at Section \ref{subsection:CG}.
\label{propMatrices}
\end{proposition}

\begin{proof}
The proof here generalizes the proof from \cite[Lemma 6.2]{Oguiso14} and \cite[Lemma 2.1]{Silverman91}. Let $j\in J$. By definition we have $\iota_j^*(h_i) = h_i$ for $i\neq j$, so we need to calculate $\iota_j^*(h_j)$.

Let $L_i$ be the divisor class on $\P(\mathbf{n})_j$ corresponding to the pullback of $\O_{\P^{n_i}}(1)$. Then $(\pi_j)_*(h_j) = \sum_{i\neq j} a_iL_i$. Using the projection formula we obtain the following equality: \begin{align*}
a_i &= \left((\pi_j)_*(h_j).L_i^{n_i-1}.\prod_{k\not \in \{i,j\}} L_k^{n_k}\right)_{\P(\mathbf{n})_j}\\
	&= \left( h_j.h_i^{n_i-1}.\prod_{k\not \in \{i,j\}} h_k^{n_k} \right)_X\\
	&= \left( H_j.H_i^{n_i-1}.\prod_{k\not \in \{i,j\}} H_k^{n_k}.(D_1.D_2.\cdots.D_n) \right)_{\P(\mathbf{n})}\\
	&= b_{ij}
\end{align*}

Now we can use that $\pi_j^*(\pi_j)_*(h_j) = h_j + \iota_j^*(h_j)$ to obtain that \[ \iota_j^*(h_j) = -h_j + \sum_{i\neq j} b_{ij} h_i. \]

Hence the matrix associated to $\iota_j^*$ is \[ \iota_j^* =  \begin{pmatrix}
1 & 0 & \cdots & 0 & b_{1\,j} & 0 & \cdots & 0 \\ 
0 & 1 & \cdots & 0 & b_{2\,j} & 0 & \cdots & 0 \\ 
\vdots & \vdots & \ddots & \vdots & \vdots & \vdots & \cdots & \vdots \\ 
0 & 0 & \cdots & 1 & b_{j-1\,j} & 0 & \cdots & 0 \\ 
0 & 0 & \cdots & 0 & -1 & 0 & \cdots & 0 \\ 
0 & 0 & \cdots & 0 & b_{j+1\,j} & 1 & \cdots & 0 \\ 
\vdots & \vdots & \ddots & \vdots & \vdots & \vdots & \ddots & \vdots \\ 
0 & 0 & \cdots & 0 & b_{l\,j} & 0 & \cdots & 1
\end{pmatrix}\] which is the transpose of $M(l)_j$.
\end{proof}

\begin{example}
Let $X$ be given by the intersection of general divisors of multidegree $(2,2,1)$, $(2,1,1)$ and $(1,1,2)$ in $\P^4\times \P^3 \times \P^3$. Here $n = 3$, and so $J = \{2,3\}$. We can write $X$ as a cycle in $\P(\mathbf{n})$.
\begin{align*}
X 	&= (2H_1 + 2H_2 + H_3) . (2H_1 + H_2 + H_3) . (H_1 + H_2 + 2H_3)\\
	&= 4H_1^3 + 2H_2^3 + 2H_3^3 + 10H_1^2H_2 + 12H_1^2H_3 + 19H_1H_2H_3 \\ 
	&\phantom{=} + 8H_1H_2^2 + 9H_1H_3^3 + 7H_2^2H_3 + 7H_2H_3^2
\end{align*}
To compute $B$ we need the coefficients $b_{12} = 8$, $b_{13} = 9$ and $b_{23} = b_{32} = 7$. The matrix $B$ is then \[ B = \begin{pmatrix}
1 & -8/2 & -9/2 \\ -8/2 & 1 & -7/2 \\ -9/2 & -7/2 & 1
\end{pmatrix} \]

To compute $\iota_2^*(h_2)$ we follow the proof of Proposition \ref{propMatrices}. Let $(\pi_2)_*(h_2) = a_1L_1 + a_3L_3$. We use intersection numbers to compute $a_3$. \begin{align*}
a_1 &= \left((\pi_2)_*(h_2) . L_1^3.L_3^3\right)_{\P^4\times \widehat{\P^3} \times \P^3}\\
	&= \left(H_2.H_1^3.H_3^3.X\right)_{\P^4 \times \P^3 \times \P^3}\\
	&= 8 = b_{12}
\end{align*}
Similarly, we get $a_3 = 7 =b_{32}$. We can repeat this to compute $\iota_3^*$, and so we obtain \[ \iota_2^* = \begin{pmatrix}
1 & 8 & 0 \\ 0 & -1 & 0 \\ 0 & 7 & 1
\end{pmatrix} \qquad ; \qquad \iota_3^* = \begin{pmatrix}
1 & 0 & 9 \\ 0 & 1 & 7 \\ 0 & 0 & -1
\end{pmatrix} \]

Notice that following Section \ref{subsection:CG} we get that the matrices associated to $\rho(s_2)$ and $\rho(s_3)$ are \[ M(3)_2 = \begin{pmatrix}
1 & 0 & 0 \\ 8 & -1 & 7 \\ 0 & 0 & 1
\end{pmatrix} \qquad \text{ and } \qquad M(3)_3 = \begin{pmatrix}
1 & 0 & 0 \\ 0 & 1 & 0 \\ 9 & 7 & -1
\end{pmatrix} \] respectively, which correspond to the transpose of $\iota_2^*$ and $\iota_3^*$.
\label{example1}
\end{example}
~\\
We will now prove that every element of $\Bir(X)$ is a composition of the involutions $\iota_j$, up to an automorphism.

\begin{proposition}[cf. \cite{CantatOguiso15}, Theorem 3.3(4)]
$\Bir(X) \simeq \Aut(X)\cdot \langle \iota_j\mid j\in J\rangle$.
\label{propBirIsoInvo}
\end{proposition}

\begin{proof}
For $j\in J$, let $X\to[\overline{\pi}_j] \overline{X_j} \to[q_j] \P(\mathbf{n})_j$ be the Stein factorization of $\pi_j$. Then $\overline{\pi}_j$ is the small contraction corresponding to the codimension $1$ face given by $\sum_{k\neq j}\R_{\geq 0} h_k$ of $\Nef(X)$. Thus, $\rho(X/\overline{X}_j) = 1$ with relative ample generator $h_j$, and so $\overline{\pi}_j$ is a flopping contraction. We can describe the flop of $\overline{\pi}_j$ explicitly.

The birational involution $\iota_j$ induces an automorphism $\overline{\iota}_j$ of $\overline{X}_j$, satisfying $\overline{\pi}_j\circ \iota_j = \overline{\iota}_j\circ \overline{\pi}_j$. Define \[ \overline{\pi}_j^+ = (\overline{\iota}_j)^{-1}\circ \overline{\pi}_j\colon X\to \overline{X}_j. \] Then $\iota_j\circ \overline{\pi}^+_j = \overline{\pi}_j$. In Proposition \ref{propMatrices} we computed that \[ \iota_j^*(h_j) = -h_j + \sum_{i\neq j} b_{ij}h_i, \] so $\iota_j^*(h_j)$ is $\overline{\pi}_j^+$-anti-ample. Hence $\overline{\pi}_j^+$ is the flop of $\overline{\pi}_j$, and $(\overline{\pi}_j^+)^{-1}\circ \overline{\pi}_j = \iota_j$.

From Theorem \ref{thmKContraction}, we have that every flopping contraction is given by a codimension 1 face of $\Nef(X)$, up to an automorphism of $X$, and since there is no codimension 1 face of $\Nef(X)$ aside from $\sum_{k\neq j} \R_{\geq 0} h_k$, $1\leq j\leq s$, it follows that the only flops are $\iota_j$, with $j\in J$. On the other hand, Theorem \ref{thmKFlops} says that the every element of $\Bir(X)$ can be decomposed as a finite sequence of flops, modulo an automorphism of $X$. Combining this two results we obtain that $\Bir(X) \simeq \Aut(X)\cdot \langle \iota_j\mid j\in J\rangle$.
\end{proof}

\begin{corollary}[Theorem \ref{mainThm} (b)]
Let $(W,S)$ be the Coxeter system with $S = \{s_i\}_{i=1}^l$ and $m_{ij} = \infty$ when $i\neq j$. Then $\langle \iota_j\mid j\in J \rangle \simeq W_J \simeq \underbrace{\Z/2\Z \ast \cdots \ast \Z/2Z}_{|J|}$.
\end{corollary}

\begin{proof}
This follows from the fact that the geometric representation of $W$ obtained with the matrix $B$ is faithful.
\end{proof}

\begin{remark}
Theorem \ref{thmKContraction} implies that if $\codim X < n$ then $X$ has no flopping contractions, and so Theorem \ref{thmKFlops} implies that $\Bir(X) = \Aut(X)$. On the other hand, Proposition \ref{propAutAction} remains true in this case, so $\Aut(X)$ is finite and acts trivially on $N^1(X)$.

The only case remaining is when $\codim X = n + 1$. This case has been partially studied (\cite[Section 3]{HosonoTakagi18}, \cite{LaiWang21}) and it seems that it has a very rich group of birational automorphisms.
\label{remarkCodimensionBir}
\end{remark}

\section{Movable cone of \texorpdfstring{$X$}{X}} \label{section:MovableCone}

Throughout this section $X$ denotes a codimension $n$, general complete intersection Calabi-Yau manifold in $\P(\mathbf{n})$ that satisfies Theorem \ref{mainThm} (1) and (2).

Let $(W,S)_B$ the Coxeter system associated to $X$ described for Proposition \ref{propMatrixDefined}. We want to give a description of ${\Mov\,}^e(X)$ using the Tits cone obtained from the geometric representation of $W$ associated to the bilinear form $B$.

To study the Tits cone we can identify the dual $V^*$ to $V$ using the bilinear form $B$. Let $\mathcal{B}(v,w) = v^TBw$ be the quadratic form defined by $B$. Then the fundamental domain $D$ can be identified with the vectors $w$ such that $\mathcal{B}(w,\alpha_i)\geq 0$ for all $i$. Define the vectors $c_i$ as \[ \mathcal{B}(c_i,\alpha_i) = 1 \] and \[ \mathcal{B}(c_i,\alpha_j) = 0\, , \, j\neq i. \] This identification also identifies the action $\rho^*$ with $\rho$.

Define $T_J$ as the orbit of $D$ under the action of $W_J$. More precisely, \[ T_J := \bigcup_{w\in W_J} \rho(w)(D). \] Notice that if $k\not \in J$, then $\rho(s_j).c_k = c_k$ for all $j\in J$.

\begin{theorem}
Let $\Phi\colon V \to N^1(X)_\R$ be the linear map such that $\Phi(c_i) = h_i$ for all $1\leq i \leq l$, and let $\Psi\colon W_J \to \Bir(X)$ given by $s_j\mapsto \iota_j$ for $j\in J$. Then
\begin{enumerate}
\item $\Phi \circ \rho(w) = \Psi(w)^*\circ \Phi$.
\item The fundamental domain $D$ of $W$ is mapped onto the the nef cone $\Nef(X)$ by $\Phi$.
\item The subcone $T_J \subset T$ is mapped onto the the movable effective cone ${\Mov\,}^e(X)$ by $\Phi$.
\end{enumerate}
\label{thmMapping}
\end{theorem}

The first statement follows from the computations of Section \ref{section:ComputingBir}. The second statement follows simply because both $D$ and $\Nef(X)$ are the convex hull of the $c_i$'s and $h_i$'s respectively. So we need to prove the last statement.

First we need to show that the subcone $T_J$ is convex.

\begin{proposition}
The cone $T_J$ is convex.
\label{propTitsConvex}
\end{proposition}

\begin{proof}
First, if $J = \{1,\ldots,l\}$, then $W_J = W$ and $T_J$ corresponds to the Tits cone which is convex.

Now assume that $J \subsetneq \{1,\ldots, l\}$. For $w \in W$, let $D_w := \rho(w)(D)$ the chamber obtained by the action of $w$ on the fundamental domain, with vertices $v_j := \rho(w)c_j$. The chamber $D_{ws_i}$ is adjacent to $D_w$, and \[ D_w \cap D_{ws_i} = \{\sum_{j\neq i} a_jv_j \mid a_j\geq 0\} \]

For each $D_w$, let $\Pi_w^i$ be the half-space defined by the hyperplane $\mathrm{Span}\{v_1,\ldots,\hat{v}_i,\ldots,v_l\}$ that contains the vector $v_i$. Notice that $D_{ws_i}$ is not contained in $\Pi_w^i$. Define \[ P:= \bigcap_{w\in W_J} \bigcap_{i\not \in J} \Pi_w^i. \]

We want to show that $T_J = T \cap P$. Because $T$ is convex, and $P$ is the intersection of half-spaces, then it follows that $T_J$ is convex.

We use the following lemma.

\begin{lemma}
If $\mathrm{int}(D_w) \cap P^c \neq \emptyset$, then $\mathrm{int}(D_w) \subseteq P^c$.
\label{lemmaChambers}
\end{lemma}

\begin{proof}
By contradiction, assume that there exists $w'\in W_J$, $i\in J$ and $x,y\in \mathrm{int}(D_w)$ such that $x\in \Pi_{w'}^i$ and $y \in (\Pi_{w'}^i)^c$.

By \cite[Lemma 1]{Vinberg71}, each segment in $T$ passes through a finite number of chambers, so we can choose $x'\in \mathrm{int}(D_{w'})$ and $y'\in \mathrm{int}(D_{w's_i})$ such that the segments $\overline{xx'}$ and $\overline{yy'}$ satisfy the following:
\begin{enumerate}
\item The segment intersects at most at one point of each codimension one face of the chambers that it passes through; and
\item the point of intersection of each face lies in the interior of the codimension one face.
\end{enumerate}

Under these conditions, the chambers that intersect the segments $\overline{xx'}$ and $\overline{yy'}$ form sequences \[ D_{w'} = D_{w_1} ,D_{w_2},\ldots ,D_{w_{m+1}} = D_w \] and \[ D_{w's_i} = D_{w'_1} ,D_{w'_2},\ldots ,D_{w'_{r+1}} = D_w \] with \[ D_{w_{j+1}} = D_{w_js_{k_j}} \qquad ;\qquad D_{w'_{j+1}} = D_{w'_js_{k'_j}} \] and $s_{k_{j+1}} \neq s_{k_j}$ and $s_{k'_{j+1}} \neq s_{k'_j}$.

\begin{figure}[h]
\begin{center}
\includegraphics[scale=1.2]{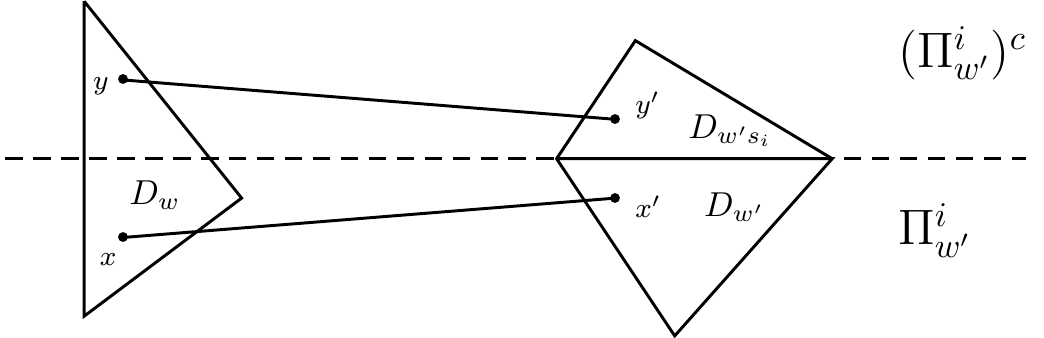}
\end{center}
\caption{An example of the situation in Lemma \ref{lemmaChambers}.}
\end{figure}

Because $\overline{xx'}\subset \Pi_{w'}^i$ and $\overline{yy'} \subset (\Pi_{w'}^i)^c$, we have that $s_{k_1} \neq s_i$ and that $s_{k'_1} \neq s_i$. Then \[ w = w'\prod_{j=1}^m s_{k_j} = w's_i\prod_{j=1}^r s_{k'_j} \] so \[\prod_{j=1}^m s_{k_j} = s_i\prod_{j=1}^r s_{k'_j}.\] But the previous condition implies that both are reduced expressions, $s_{k_1} \neq s_i$, contradicting the fact that $W$ is a free product of $l$ copies of $\Z/2\Z$.

Therefore either $\mathrm{int}(D_w) \subset P$ or $\mathrm{int}(D_w) \subset P^c$.
\end{proof}

Back to the proof of Proposition \ref{propTitsConvex}. Lemma \ref{lemmaChambers} implies that it is enough to show that a chamber $D_w\subset P^c \cap T$ if and only if $w\not \in W_J$. To do so, we use the same segment argument used in the proof of Lemma \ref{lemmaChambers}.

First, assume that $D_w\subset P^c \cap T$, so $D_w \subset (\Pi_{w'}^i)^c$ for some $w'\in W_J$ and $i\not \in J$. Then there exists a segment $\overline{xy}$ with $x\in D_w$ and $y\in D_{w's_i}$ totally contained in $(\Pi_{w'}^i)^c$. Then \[ w = w's_i\prod_{j=1}^r s_{k'_j} \] implying that $w\not \in W_J$.

Assume now that $w\not \in W_J$. We can write $w = w's_iw''$, with $w' \in W_J$. If $D_w \subset \Pi_{w'}^i$, then there exists a segment $\overline{xy}$ with $x\in D_w$ and $y\in D_{w'}$. Then $w = w's_{k_1}\prod_{j=2}^r s_{k_j}$, with $s_{k_1}\prod_{j=2}^r s_{k_j}$ a reduced expression, and $s_{k_1} \neq s_i$, which is a contradiction.

Hence $D_w\subset P^c \cap T$ if and only if $w\not \in W_J$.

To finish the proof, let $w'\in W_J$. We just proved that $D_{w'}\not \subset P^c$, so by Lemma \ref{lemmaChambers} $D_{w'} \subset P$, implying that $T_J = T \cap P$.
\end{proof}

Now, because the divisor classes $h_j$ are semiample (and therefore movable) and they generate $\Nef(X)$, we have that $\Nef(X) \subset {\Mov\,}^e(X)$. Because the movable effective cone is invariant under the action of $\Bir(X)$, we have that the orbit of the nef cone $\Nef(X)$ is contained in ${\Mov\,}^e(X)$, so \[ \Phi(T_J) \subseteq {\Mov\,}^e(X). \]

We then need to show the following proposition.

\begin{proposition} Let $X$ be a general complete intersection Calabi-Yau manifold in $\P(\mathbf{n})$ that satisfies Theorem \ref{mainThm} (1) and (2). 
\begin{enumerate}
\item For all $g\in \Bir(X)$ such that $g^* \neq \mathrm{Id}$, $g^*(\Nef(X)) \cap \Nef(X) = \emptyset$.
\item ${\Mov\,}^e(X) \subseteq \Phi(T_J)$.
\end{enumerate}
\label{propIsoTitsNef}
\end{proposition}

\begin{proof}
If $g\in \Bir(X)$ satisfies $g^*(\Nef(X)) \cap \Nef(X) \neq \emptyset$, then \cite[Lemma 1.5]{Kawamata97} implies that $g\in \Aut(X)$. From Proposition \ref{propAutAction} we have that $\Aut(X)$ acts trivially on $N^1(X)_\R$, so $g^* = \mathrm{Id}$, which completes the proof of (1).

To prove (2) we need the following lemma.

\begin{lemma}[cf. \cite{CantatOguiso15}, Lemma 4.4]
For any given pseudoeffective integral divisor class $E$ of $X$, there is a birational automorphism $g$ of $X$ such that $g^*(E) \in \Nef(X)$
\label{lemmaEffectiveComputation}
\end{lemma}

\begin{proof}
Define $E_1 := E$. In $N^1(X)_\R$ we can write \[ E_1 = \sum_{i=1}^l \beta_i(E_1)h_i, \] where the $\beta_i(E_1)$ are integers. Set \[ s(E_1):= \sum_{i=1}^l \beta_i(E_1). \]

Since $E_1$ is pseudoeffective, and the $h_i$ are nef, we get the following facts: 
\begin{enumerate}
\item $\beta_k (E_1)\geq 0$ for $k\not \in J$. This follows from the fact that $n_k > n$ and \[ \left(\prod_{i=1}^n a_{ik} \right)\beta_k(E_1) = \left(E_1.h_k^{n_k-(n+1)}.\prod_{i\neq k} h_i^{n_i}\right)_X \geq 0. \]
\item At most one $\beta_j(E_i) < 0$ for $j\in J$. To get this, notice that for $j_1,j_2\in J$ \[ \left(E_1.h_{j_1}^{n-1}.\prod_{i\neq j_1,j_2} h_i^{n_i}\right)_X = 2\beta_{j_1}(E_1) + b_{j_1,j_2}\beta_{j_2}(E_1) \geq 0 \] and \[ \left(E_1.h_{j_2}^{n-1}.\prod_{i\neq j_1,j_2} h_i^{n_i}\right)_X = 2\beta_{j_2}(E_1) + b_{j_2,j_1}\beta_{j_1}(E_1) \geq 0. \] From Proposition \ref{propMatrixDefined} we have that $b_{j_1,j_2} = b_{j_2,j_1}\geq 0$, so $\beta_{j_1} + \beta_{j_2} \geq 0$. This implies that at most one $\beta_j(E_1) <0$ for $j\in J$.
\end{enumerate}

Even more, this implies that $s(E_1)$ is non-negative. Say for example that $1,2\in J$ and that $\beta_1(E_1)<0$. Then \[ s(E_1) = (\beta_1(E_1) + \beta_2(E_2)) + \beta_3(E_1) +\cdots + \beta_l(E_1) \geq 0 \]

If $E_1\in \Nef(X)$ we can take $g=\mathrm{Id}$.  So we can assume that $E_1$ is not nef, and so that $\beta_j(E_1) <0$ for some $j\in J$.

Consider the divisor class $E_2 := \iota_j^*(E_1)h_j$. It is effective and so $s(E_2)\geq 0$, and by the matrices from Proposition \ref{propMatrices} we have \[ E_2 = - \beta_j(E_1)h_j + \sum_{i\neq j} (\beta_i(E_1) + b_{ij}\beta_j(E_1))h_i. \] We can compute $s(E_2)$ to obtain
\begin{align*}
s(E_2) &= -\beta_j(E_1) + \sum_{i\neq j} (\beta_i(E_1) + b_{ij}\beta_j(E_1))\\
	&= s(E_1) - 2\beta_j(E_1) + \sum_{i\neq j} b_{ij}\beta_j(E_1) \\
	&< s(E_1),
\end{align*} where the last inequality follows from the fact that the $b_{ij} \geq 2$ and $\beta_j(E_1) < 0$.

If all the $\beta_i(E_2)$ are non-negative, then $E_2\in\Nef(X)$ and we are done. If not, there is a unique $j'\in J$ such that $\beta_{j'}(E_2) <0$. Consider $E_3 = \iota_{j'}^*(E_2)$. Then, as above, we get that $E_3$ is effective and $s(E_3) < s(E_2)$. At each step of this process $s(\cdot)$ decreases, and because $s$ is non-negative, this process stops. This implies that there exists an effective divisor $E_k$ in the $\Bir(X)$-orbit of $E_1$ such that $E_k$ is nef.

\end{proof}

\begin{example}
As in Example \ref{example1}, let $X$ be given by the intersection of general divisors of multidegree $(2,2,1)$, $(2,1,1)$ and $(1,1,2)$ in $\P^4\times \P^3 \times \P^3$. Let $E_1 = \beta_1(E_1)h_1 + \beta_2(E_1)h_2 + \beta_3(E_1)h_3$ a pseudoeffective integral divisor class. Then \[ 4\beta_1(E_1) = (E_1.h_1.h_2^3.h_3^3)_X \geq 0 \] and \begin{align*}
2\beta_2(E_1) + 7\beta_3(E_1) &= (E_1.h_1^4.h_2^2)_X \geq 0 \\
7\beta_2(E_1) + 2\beta_3(E_1) &= (E_1.h_1^4.h_3^2)_X \geq 0
\end{align*}
so $\beta_2(E) + \beta_3 \geq 0$. Assume that $\beta_2(E_1) < 0$. Then we compute $E_2 = \iota_2^*(E_1)$
\begin{align*}
E_2 &= \beta_1(E_1)h_1 + (8\beta_2(E_1)h_1 - \beta_2(E_1)h_2 + 7\beta_2(E_1)h_3) + \beta_3(E_1)h_3\\
	&= (\beta_1(E_1) + 8\beta_2(E_1))h_1 - \beta_2(E_1)h_2 + (\beta_3(E_1) + 7\beta_2(E_1))h_3
\end{align*}
and so
\begin{align*}
s(E_2) 	&= (\beta_1(E_1) + 8\beta_2(E_1)) - \beta_2(E_1) + (\beta_3(E_1) + 7\beta_2(E_1))\\
		&= \beta_1(E_1) + 14\beta_2(E_1) + \beta_3(E_1)\\
		&= s(E_1) + 13\beta_2(E_1)\\
		&< s(E_1).
\end{align*}
\end{example}
~\\
To finish the proof of Proposition \ref{propIsoTitsNef}, let $u$ be an element of ${\Mov\,}^e(X)$. We can write \[ u = \sum_{i=1}^m r_iE_i \] with $r_i$ positives real numbers and $E_i$ effective divisor classes. From the previous lemma we have that each $E_i$ is in the image of $T_J$, and because $T_J$ is convex we have that $u\in \Phi(T_J)$, and so ${\Mov\,}^e(X) \subseteq \Phi(T_J)$.
\end{proof} 

\begin{corollary}[Theorem \ref{mainThm} (c)]
Conjecture \ref{ConjectureKM} holds for $X$.
\end{corollary}

\begin{remark}
Performing a similar computation as in Lemma \ref{lemmaEffectiveComputation} we can check that if $\codim X < n$, then every pseudoeffective integral divisor class $E$ of $X$ is nef, implying that $\Nef(X) = \Mov(X)$. This together with Remark \ref{remarkCodimensionBir} implies that Theorem \ref{mainThm} is true when $\codim X < n$.
\label{remarkCodimensionMov}
\end{remark}

The goal now is to describe the boundary of the movable cone $\Mov(X)$. The Kawamata-Morrison conjecture gives a description of ${\Mov\,}^e(X)$ in terms of the action of $\Bir(X)$, but it is not clear what happens at the boundary of the movable cone.

From Theorem \ref{thmMapping} (3), the problem of describing the boundary of the movable cone is equivalent to describing the boundary of the subcone $T_J$ of the Tits cone $T$. In general, this is not an easy task, but we can give a precise description when $W_J$ is Lorentzian.

Then, in order to prove Theorem \ref{mainThm2}, we show the equivalent statement in $T_J \subset V$.

\begin{theorem}
Suppose that the subgroup $W_J$ is Lorentzian. The boundary of the cone $\overline{T_J}$ is the closure of the union of the following sets:
\begin{enumerate}
\item The $W_J$-orbit of the codimension one faces $\{ \sum_{k\neq i} a_kc_k \mid a_k\geq 0\}$, for $i\not \in J$; and
\item The $W_J$-orbit of the cones $\{ a_\lambda v_\lambda + \sum_{k\neq i,j} a_kc_k \mid a_k\geq 0, a_\lambda \geq 0\}$, for $i,j\in J$, where $v_\lambda$ is:
\begin{itemize}
\item If $n \geq 2$, an eigenvector associated to the unique eigenvalue $\lambda > 1$ of $\rho(s_is_j)$; or
\item If $n = 1$, $v_\lambda=0$.
\end{itemize}
\end{enumerate}
\label{thmMovStructure}
\end{theorem}

First we study the eigenvalues of the matrices involved in the theorem.

\begin{proposition}
For $i \neq j$, the matrix $\rho(s_js_i)$ has one eigenvalue $\lambda >1$ if $n\geq 2$, and if $n=1$, then it is not diagonalizable and 1 is its only eigenvalue.
\label{largeEigenvalue}
\end{proposition}

\begin{proof}
First notice that the matrix $(\iota_j\iota_i)^*$ is the transpose of $\rho(s_js_i)$, so it is enough to show the result for $(\iota_j\iota_i)^*$.

From Proposition \ref{propMatrices}, we can write the matrices $\iota_i^*$ and $\iota_j^*$ in term of their columns as \[ \iota_i^* = (e_1\, e_2 \, \cdots\, e_{i-1} \, v_i \, e_{i+1}\, \cdots \, e_l) \] and \[ \iota_j^* = (e_1\, e_2 \, \cdots\, e_{j-1} \, v_j \, e_{j+1}\, \cdots \, e_l) \] where\[v_i = (b_{1i},b_{2i},\ldots,b_{i-1\, i},-1,b_{i+1\, i},\ldots ,b_{li}),\] and \[v_j = (b_{1j},b_{2j},\ldots,b_{j-1\, i},-1,b_{j+1\, j},\ldots ,b_{lj}).\]

Then, by computing the product of these matrices, the characteristic polynomial of $(\iota_j\iota_i)^* = \iota_i^*\iota_j^*$ is \[ (x-1)^{l-2}(x^2 - b_{ij}^2 x +2x + 1). \]

If $n=1$, then $b_{ij} = 2$, and so the characteristic polynomial is $(x-1)^l$, but because $(\iota_j\iota_i)^*$ is not the identity matrix, $(\iota_j\iota_i)^*$ is not diagonalizable.

If $n\geq 2$, then $b_{ij} > 2$, and so $(x^2 - b_{ij}^2 x +2x + 1)$ has two different real roots $\lambda_1,\lambda_2$. Because $\lambda_1\lambda_2 = 1$, then one of them is greater than 1.

\end{proof} 

Notice that the vectors $c_k$ with $k\neq i,j$ are the eigenvectors corresponding to the eigenvalue $1$ of $\rho(s_js_i)$. 

\begin{proposition}
Assume that $n\geq 2$. Let $v_\lambda$ be an eigenvector associated to the largest eigenvalue $\lambda > 1$ of $\rho(s_js_i)$, such that $v_\lambda \in \overline{T}_J$. Then the convex cone generated by the set $\{c_k\mid k\neq i,j\} \cup \{v_\lambda\}$ is a face of the boundary of $\overline{T}_J$.
\label{propConesBoundary}
\end{proposition}

\begin{proof}
Given that $\overline{T}_J$ is invariant under multiplication by elements of $W_J$, it follows from the Perron-Frobenius-Birkhoff theorem \cite{Birkhoff67} that such $v_\lambda$ exists in $\overline{T}_J$.

Let \[ v := b_1 v_\lambda + \sum_{k\neq i,j} a_kc_k \quad ; \qquad a_k \geq 0, b_1\geq 0 \] be an arbitrary vector on $C:= \mathrm{cone}( \{c_k\mid k\neq i,j\} \cup \{v_\lambda\})$. From Proposition \ref{largeEigenvalue}, we have that the set $\{c_k\}_{k\neq i,j}\cup \{v_\lambda,v_{\lambda^{-1}}\}$ is a basis of $V$. Even more, if all the coefficients of a linear combination of these vectors are positive, then the vector is in $\overline{T}_J$.

Fix $\epsilon > 0$, and let $\delta$ be a positive number. Consider $b_2 \gg 0$ such that \[ \sum_{k\neq i,j} (a_k + \delta)c_k + (b_1 + \delta) v_\lambda - b_2 v_{\lambda^{-1}}\] is not in $\overline{T}_J$. Let $n \gg 0$ such that $\lambda^{-n}b_2 < \delta$.

Let \begin{align*}
v_\delta &:= \rho(s_js_i)^n\left( \sum_{k\neq i,j} (a_k + \delta)c_k + \lambda^{-n}(b_1 + \delta) v_\lambda + b_2 v_{\lambda^{-1}} \right)\\ 
&=  \sum_{k\neq i,j} (a_k + \delta)c_k + (b_1 + \delta) v_\lambda + \lambda^{-n}b_2 v_{\lambda^{-1}}
\end{align*} and \begin{align*}
v_{-\delta} &:= \rho(s_js_i)^n\left( \sum_{k\neq i,j} (a_k + \delta)c_k + \lambda^{-n}(b_1 + \delta) v_\lambda - b_2 v_{\lambda^{-1}} \right)\\ 
&=  \sum_{k\neq i,j} (a_k + \delta)c_k + (b_1 + \delta) v_\lambda - \lambda^{-n}b_2 v_{\lambda^{-1}}.
\end{align*} Then for a constant $p$, and $\delta < p\epsilon$, $v_{\delta}$ and $v_{-\delta}$ are in the ball of radius $\epsilon$ centered at $v$, but $v_{\delta}$ is in the interior of $\overline{T}_J$, and $v_{-\delta}$ is not in $\overline{T}_J$, so $v$ is on the boundary of $\overline{T}_J$.
\end{proof}

\begin{proof}[Proof of Theorem \ref{thmMovStructure}.]

Recall that from the proof of Proposition \ref{propTitsConvex} we have that $T_J = T \cap P$, so the $W_J$-orbit of the codimension one faces $\{ \sum_{k\neq i} a_kc_k \mid a_k\geq 0\}$, for $i\not \in J$, are on the boundary of $\overline{T}_J$.

The rest of the boundary corresponds to accumulations of chambers $D_w$ as follows. Let $x\in \partial \overline{T}_J$. Then there exists $\{D_{w_i}\}$ an injective sequence of chambers, with $w_i\in W_J$, such that there exist $x_i \in D_{w_i}$, with $x_i \to x$. We can assume that the $w_i$ corresponds to a reduced word in $W_J$, and even more we can assume, by possibly extending the sequence, that the sequence of chambers is given by adjoint chambers. This means that $w_i = s_{k_1}s_{k_2}\cdots s_{k_i}$. So from each sequence of chambers we obtain an infinite reduced word $\mathbf{w} = s_{k_1}s_{k_2}s_{k_3}\cdots$, and because $W_J$ is a free product of $\Z/2\Z$ this correspondence is unique.

For some $m>0$, let $\{\widehat{v_{k_1}},\ldots, \widehat{v_{k_r}}\}$, where $v_{k_j} = w_m.c_{k_j}$, be the set of vertices fixed by all $w_i$, for $i\geq m$. This set might be empty, and it cannot have more than $l-2$ elements. From Theorem \ref{thmInfiniteWord}, the chambers $D_{w_i}$ converge to the cone $C$ generated by the $\{v_{k_1},\ldots,v_{k_r}\}\cup \{\gamma(\mathbf{w})\}$.

Let $C'$ be one of the cones $\{ a_\lambda v_\lambda + \sum_{k\neq i,j} a_kc_k \mid a_i\geq 0\}$, such that $\{v_{k_1},\ldots,v_{k_r}\}\subseteq \{w_m.c_k\mid k\neq i,j\}$. Then again, by Theorem \ref{thmInfiniteWord}, we have that the sequence of cones $\{w_i.C'\}$ converges to $C$. Therefore $C'$ is on the closure of the $W_J$-orbit of $C$, and because $C$ is on the boundary of $\overline{T_J}$, by Proposition \ref{propConesBoundary}, then so is $C'$.

\end{proof}

The fact that the subgroup $W_J$ is Lorentzian is strongly used in the proof of Theorem \ref{mainThm2}. If $|J| =2$ and $n\geq 2$, or $|J| = 3$, then $W_J$ is Lorentzian. If $|J| = 4$, we cannot conclude this, but using arguments from \cite{Krammer09}, the proof of Theorem \ref{mainThm2} remains true. Also, when $n=1$ this is proven in \cite{CantatOguiso15} for $|J|\geq 3$. If $|J| = 2$ and $n=1$, we have that $\rho(X) \geq 3$, so $W_J$ can be thought as a subgroup of a Lorentzian group, so Theorem \ref{thmInfiniteWord} remains true.  Using the help of a computer, we checked that for $n < 500$ and $|J| < 50$ the subgroup $W_J$ is Lorentzian. It is interesting to see if this property holds for all $n$ and all $|J|$.

\begin{conjecture}
Let $X$ be a general complete intersection Calabi-Yau manifold in $\P(\mathbf{n})$ that satisfies Theorem \ref{mainThm}. Then $W_J$ is Lorentzian.
\label{conjectureLorentzian}
\end{conjecture}

Proving Conjecture \ref{conjectureLorentzian} is equivalent to proving the following.

\begin{proposition}
Let $n\geq 2$ and let $0 \leq r_1\leq r_2\leq \ldots \leq r_n$ be a partition of $|J|$, \emph{i.e.}, $\sum r_i = |J|$. Let $A_r$ be a $r\times r$ matrix given by \[ (A_r)_{ij} = \begin{cases} 1 & \text{if } i=j \\ -n & \text{if } i\neq j \end{cases}. \] Let $B_J$ be the matrix formed by the blocks $A_{r_k}$ on the diagonal, and outside of it the entries of the matrix only have the value $-(2n+1)/2$. Then Conjecture \ref{conjectureLorentzian} is true if and only if the signature of $B_J$ is $(|J|-1,1)$ for all $n$, for all $|J|$, and for all partitions $0 \leq r_1\leq r_2\leq \ldots \leq r_n$ of $|J|$.
\label{propEquivalenceConjecture}
\end{proposition}

\begin{example}
Let $X$ be the intersection of general divisors of multidegree $(3,1,2,1,1,2,1)$, $(2,2,1,2,2,1,1)$ and $(2,3,1,1,1,1,2)$ in $\P^6 \times \P^5 \times \P^3 \times \P^3 \times \P^3 \times \P^3 \times \P^3$. Then the matrix $B_J$ is \[ B_J = \begin{pmatrix}
1 & -7/2 & -7/2 & -3 & -7/2 \\ 
-7/2 & 1 & -3 & -7/2 & -7/2 \\
-7/2 & -3 & 1 & -7/2 & -7/2 \\
-3 & -7/2 & -7/2 & 1 & -7/2 \\
-7/2 & -7/2 & -7/2 & -7/2 & 1
\end{pmatrix} \]
By swapping rows and columns 1 and 5 we get that $B_J$ is similar to the matrix \[ \begin{pmatrix}
1 & -7/2 & -7/2 & -7/2 & -7/2 \\ 
-7/2 & 1 & -3 & -7/2 & -7/2 \\
-7/2 & -3 & 1 & -7/2 & -7/2 \\
-7/2 & -7/2 & -7/2 & 1 & -3 \\
-7/2 & -7/2 & -7/2 & -3 & 1
\end{pmatrix}, \] which is the matrix obtained in Proposition \ref{propEquivalenceConjecture} with $r_1 = 1$ and $r_2 = r_3 = 2$. In this case, the eigenvalues of $B_J$ are \[ \mathrm{Eigenvalues} = \{ -4 - \sqrt{74}, 5 ,-4 + \sqrt{74}, 4 ,4 \} \] and so the signature is $(4,1)$, which implies that $W_J$ is Lorentzian.
\end{example}
~\\
Using Proposition \ref{propEquivalenceConjecture}, it is easy to check that Conjecture \ref{conjectureLorentzian} is true when $r_n = |J|$ and when $|J| = n$ and all $r_i = 1$.

\section{Application: Numerical dimension of divisors} \label{section:NumericalDimension}

In \cite{Lesieutre19} and \cite{Mccleerey19} it is proven that the different notions of numerical dimension of a divisor might not coincide. In particular, two new numerical dimensions are introduced for a divisor $D$: $\kappa_\sigma^\R(D)$ and $\nu_{\text{vol}}^\R(D)$. We focus on the latter.

\begin{definition}
Suppose that $X$ is a projective smooth variety and $D$ a pseudoeffective divisor class on $X$. Fix an ample divisor $A$. We define $\nu^\R_{\text{vol}}(D)$ as the largest real number $k$ for which there is a constant $C$ such that \[ Ct^{\dim X - k}< \mathrm{vol}(tD+A) \] for all $t > 0$.
\end{definition}

From \cite{Lesieutre19} we have a formula to compute $\nu^\R_{\text{vol}}(D)$.

\begin{proposition}[\cite{Lesieutre19}, Lemma 8]
Suppose that $\phi\colon X \dashrightarrow X$ is a pseudoautomorphism such that the spectral radius of $\phi^*$ is greater than 1. Let $\lambda_1$ be the spectral radius of $\phi^*$ and $\mu_1$ the spectral radius of ${\phi^{-1}}^*$. Suppose that there exist a $\lambda_1$-eigenvector $\Delta_+$ for $\phi^*$ and a $\mu_1$-eigenvector $\Delta_-$ for ${\phi^{-1}}^*$ with the property that $A = \Delta_+ + \Delta_-$ is ample. Then
\[ \nu_{\text{vol}}^\R(\Delta_+) = (\dim X)\left( 1 + \dfrac{\log \mu_1}{\log \lambda_1}\right)^{-1} \]
\label{volLes}
\end{proposition}

For $X$ as in Section \ref{section:MovableCone}, because of Theorem \ref{mainThm} (3), we just need $A = \Delta_+ + \Delta_-$ to be big. From the proof of Proposition \ref{propTitsConvex} and Theorem \ref{thmInfiniteWord}, if $W_J$ is Lorentzian, then each $g\in \Bir(X)$ with spectral radius greater than 1 satisfies the hypothesis of Proposition \ref{volLes}. Even more, we have the following proposition regarding the eigenvalues of $g^*$.

\begin{proposition}
Suppose that $W_J$ is Lorentzian. If $g\in \Bir(X)$ has spectral radius $\lambda>1$, then the eigenvalues of $g^*$ are $\lambda,\lambda^{-1}$ and the rest of the eigenvalues have modulus equal to 1.
\label{volEigen}
\end{proposition}

\begin{proof}
The existence of an eigenvalue equal to $\lambda$ is again a consequence of the Perron-Frobenius-Birkhoff theorem \cite{Birkhoff67}. Because $g^*$ fixes the divisor classes $h_i$, with $i\not \in J$, then $g^*$ has $l-|J|$ eigenvalues equal to 1. On the other hand, the rest of the eigenvalues are equal to the eigenvalues of the minor obtained by deleting the rows and columns with index not in $J$. The obtained matrix corresponds to an isometry of a Lorentz space, so \cite[Section 3.7]{Riesz58} implies the rest of the proposition.
\end{proof}

\begin{corollary}
For $m\geq 3$, let \[ \mathfrak{V}_m = \{ \nu^\R_{\text{vol}}(D)\mid \dim X = m,\ D \text{ is a pseudoeffective divisor class on }X\}.\] Then $(\dim X)/2 \in \mathfrak{V}_m$.
\end{corollary}

\begin{proof}
From \cite[Section 2.2.2]{CantatOguiso15} we have that if $n=1$, $W_J$ is Lorentzian. Also, by \cite{Krammer09}, we have that the composition of all the involutions of $X$ satisfies the hypothesis of Proposition \ref{volLes}. By  Proposition \ref{volEigen}, $\lambda_1 = \mu_1$, so \[ \nu_{\text{vol}}^\R(\Delta_+) = (\dim X)\left( 1 + 1\right)^{-1} = \frac{\dim X}{2}. \]
\end{proof}

\begin{corollary}
If $|J| \geq 3$, then there are infinitely many divisor classes $D$ on $X$ with $\nu_{\text{vol}}^\R(D) = (\dim X)/2$.
\end{corollary}

\begin{proof}
If we restrict to a subgroup of $W_J$ generated by 3 generators it is Lorentzian, and the set of elements $w$ with eigenvalue $\lambda > 1$ is not empty (see the proof of \cite[Theorem 3.10]{ChenLabbe17}). This implies that there exists a divisor class $D$ with $\nu_{\text{vol}}^\R(D) = (\dim X)/2$ from Proposition \ref{volEigen}, and so the infinitely many elements of the $\Bir(X)$-orbit of $D$ we also have $\nu_{\text{vol}}^\R = (\dim X)/2$.
\end{proof}

\begin{remark}
If Conjecture \ref{conjectureLorentzian} is true, this would imply that $(\dim X)/2$ is the only interesting value that can be obtained for $\nu_\text{vol}^\R$ using Proposition \ref{volLes} on our $X$. On the other hand, a counterexample for Conjecture \ref{conjectureLorentzian} would very likely generate a divisor class $\Delta$ with irrational numerical dimension $\nu_\text{vol}^\R(\Delta)$.
\end{remark}

\begin{remark}
Proposition \ref{volLes} says nothing about the numerical dimension of the divisor classes obtained as the accumulation point of a general infinite reduced word $\mathbf{w}$. In general it is not true that the numerical dimension has any type of semi-continuity, making it hard to see a relation between this computation and the one presented in Proposition \ref{volLes}.  This could be a source of interesting numerical dimension values.
\end{remark}

\bibliography{mybib}{}
\bibliographystyle{alpha}
\end{document}